\documentclass[11pt]{article}
\usepackage{mathrsfs}
\usepackage{bbm}
\usepackage{amsfonts}
\usepackage{amsmath,amssymb}
\usepackage{amsmath}
\usepackage{amssymb}
\usepackage{graphicx}

\newtheorem{theorem}{Theorem}[section]
\newtheorem{corollary}[theorem]{Corollary}
\newtheorem{definition}[theorem]{Definition}
\newtheorem{example}[theorem]{Example}
\newtheorem{lemma}[theorem]{Lemma}
\newtheorem{proposition}[theorem]{Proposition}
\newtheorem{remark}[theorem]{Remark}

\newenvironment{proof}[1][Proof]{\noindent \textbf{#1.} }{\  $\Box$}
\numberwithin{equation}{section}

\begin{document}

\title{\textbf{Stein Type Characterization for $G$-normal Distributions }}
\author{Mingshang Hu \thanks{%
Qilu Institute of Finance, Shandong University, Jinan, China. humingshang@sdu.edu.cn. Research supported by NSF
(No.11201262 and 11301068) and Shandong Province (No.BS2013SF020 and
ZR2014AP005). } \and %
 Shige
Peng\thanks{%
School of Mathematics and Qilu Institute of Finance, Shandong University, Jinan, China. peng@sdu.edu.cn. Research supported by NSF (No.11526205 and 11221061) and by the 111
Project (No. B12023).} \and Yongsheng Song\thanks{%
Academy of Mathematics and Systems Science, CAS, Beijing, China,
yssong@amss.ac.cn. Research supported  by NCMIS;
Key Project of NSF (No. 11231005); Key Lab of Random Complex
Structures and Data Science, CAS (No. 2008DP173182).}  }
\maketitle
\date{}

\begin{abstract} In this article,  we provide a Stein type characterization for $G$-normal distributions:  Let $\mathcal{N}[\varphi]=\max_{\mu\in\Theta}\mu[\varphi],\ \varphi\in C_{b,Lip}(\mathbb{R}),$ be a sublinear expectation. $\mathcal{N}$ is $G$-normal if and only if  for any $\varphi\in C_b^2(\mathbb{R})$, we have
 \[\int_\mathbb{R}[\frac{x}{2}\varphi'(x)-G(\varphi''(x))]\mu^\varphi(dx)=0,\] where $\mu^\varphi$ is  a  realization of $\varphi$ associated with $\mathcal{N}$, i.e., $\mu^\varphi\in \Theta$ and $\mu^\varphi[\varphi]=\mathcal{N}[\varphi]$.

\end{abstract}


\textbf{MSC-classification}: 35K55, 60A05, 60E05


\section{Introduction}
Peng (2007) introduced the notion of $G$-normal distribution via the viscosity solutions of the $G$-heat equation below
\begin {eqnarray*}
\partial_t u-G(\partial^2_x u)&=&0, \ (t,x)\in (0,\infty)\times \mathbb{R},\\
                        u(0,x)&=& \varphi (x),
\end {eqnarray*} where $G(a)=\frac{1}{2}(\overline{\sigma}^2 x^+-\underline{\sigma}^2 x^-)$, $a\in \mathbb{R}$ with $0\leq\underline{\sigma}\leq\overline{\sigma}<\infty$, and $\varphi\in C_{b,Lip}(\mathbb{R})$, the collection of bounded Lipstchiz functions on $\mathbb{R}$.

 Then the one-dimensional $G$-normal  distribution is defined by \[\mathcal{N}_G[\varphi]=u^\varphi(1,0),\] where $u^\varphi$ is the  viscosity solution to the $G$-heat equation with the initial value $\varphi$.

As is well known,  the fact that $\mu=\mathcal{N}(0, \sigma^2)$ if and only if
\begin {eqnarray}\label {SteinF}\int_{\mathbb{R}}[x\varphi'(x)-\sigma^2\varphi''(x)]\mu(dx), \ \textmd{for all} \ \varphi\in C_b^2(\mathbb{R}).
\end {eqnarray} This is the characterization of the normal
distribution presented in Stein (1972), which  is the basis of  Stein's method for normal approximation (see Chen, Goldstein and Shao (2011) and the references therein for more details).

What is the proper counterpart of (\ref {SteinF}) for $G$-normal distributions? An immediate conjecture should be \[\mathcal{N}_G[\mathcal{L}_G\varphi]=0, \ \textmd{for all} \ \varphi\in C_b^2(\mathbb{R}),\] where $\mathcal{L}_G\varphi(x)=\frac{x}{2}\varphi'(x)-G(\varphi''(x)).$ However, the above equality does not hold generally as was pointed in Hu et al (2015) by a counterexample.

By calculating some examples, we try to find the proper generalization of  $(\ref {SteinF})$ for $G$-normal distributions. As a sublinear expectation,  the $G$-normal distribution can be represented as
 \[\mathcal{N}_G[\varphi]=\max_{\mu\in\Theta_G}\mu[\varphi], \ \textmd{for all} \  \varphi\in C_{b,Lip}(\mathbb{R}),\] where $\Theta_G$ is a set of probabilities on $\mathbb{R}$. For $\varphi\in C_{b,Lip}(\mathbb{R})$, we call $\mu\in\Theta_G$ a \textit{realization} of $\varphi$ associated with $\mathcal{N}_G$ if $\mathcal{N}_G[\varphi]=\mu[\varphi].$
\begin {example} \label {E1} Set $\beta=\frac{\overline{\sigma}}{\underline{\sigma}}$ and $\sigma=\frac{\overline{\sigma}+\underline{\sigma}}{2}$. Song (2015) defined a periodic function $\phi_\beta$ as a variant of the trigonometric function $\cos x$ (see Figure 1).
\begin {equation}\phi_\beta(x)=
\begin {cases}\frac{2}{1+\beta}\cos (\frac{1+\beta}{2}x) & \textmd{for $x\in[-\frac{\pi}{1+\beta},\frac{\pi}{1+\beta})$;}\\
\frac{2\beta}{1+\beta}\cos (\frac{1+\beta}{2\beta}x+\frac{\beta-1}{2\beta}\pi) & \textmd{for $x\in[\frac{\pi}{1+\beta},\frac{(2\beta+1)\pi}{1+\beta})$.}
\end {cases}
\end {equation}
\begin {figure} [!ht]
\centering
\includegraphics[scale=1]{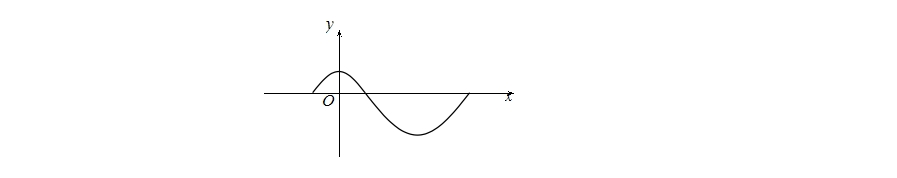}
\caption {$\phi_\beta(x)$}
\end {figure}

\noindent It was proved that \[G(\phi''_\beta(x))=-\frac{\sigma^2}{2}\phi_\beta(x)\] and that $u(t,x):=e^{-\frac{1}{2}\sigma^2t}\phi_\beta(x)$  is a solution to the $G$-heat equation. Therefore
\begin {eqnarray}\label {exam1}u(t,x)=\mathcal{N}_G[\phi_\beta(x+\sqrt{t}\cdot)]=\mu^{t,x}[\phi_\beta(x+\sqrt{t}\cdot)],
 \end {eqnarray} where $\mu^{t,x}$ is a realization of $\phi_\beta(x+\sqrt{t}\cdot)$. Since $\mu^{s,x}[\phi_\beta(x+\sqrt{t}\cdot)]$ considered as a function of $s$ attains its maximum at $s=t$, we get
 \[\partial_tu(t,x)=\int_\mathbb{R}\phi'_\beta(x+\sqrt{t}y)\frac{y}{2\sqrt{t}}\mu^{t,x}(dy)=\frac{1}{t}\int_\mathbb{R}\partial_y\phi_\beta(x+\sqrt{t}y)\frac{y}{2}\mu^{t,x}(dy)\] by taking formal  derivation on the equality (\ref {exam1}) with respect to $t$. On the other hand, noting  $u(t,x)=e^{-\frac{1}{2}\sigma^2t}\phi_\beta(x)$, we have
  \begin {eqnarray*}\partial_t u(t,x)&=&-\frac{1}{2}\sigma^2u(t,x)\\
 &=&-\frac{1}{2}\sigma^2\int_\mathbb{R}\phi_\beta(x+\sqrt{t}y)\mu^{t,x}(dy) \\
 &=&\int_\mathbb{R}G(\phi''_\beta(x+\sqrt{t}y))\mu^{t,x}(dy)\\
 &=&\frac{1}{t}\int_\mathbb{R}G(\partial^2_y\phi_\beta(x+\sqrt{t}y))\mu^{t,x}(dy).
  \end {eqnarray*} Combing the above arguments, we get the following equality
  \[\int_\mathbb{R}[\frac{y}{2}\partial_y\phi_\beta(x+\sqrt{t}y)-G(\partial^2_y\phi_\beta(x+\sqrt{t}y))]\mu^{t,x}(dy)=0.\]
\end {example}
Inspired by  this example, we guess generally the following result  holds.
\begin {proposition} \label {main} Let $\varphi\in C^2_b(\mathbb{R})$. If $\mu^\varphi$ is a realization of $\varphi$ associated with the $G$-normal distribution $\mathcal{N}_G$, we have \[\int_\mathbb{R}[\frac{x}{2}\varphi'(x)-G(\varphi''(x))]\mu^\varphi(dx)=0.\]
\end {proposition} To convince ourselves, let us calculate another simple example.

\begin {example} \label {E2} Let $\phi\in C^2(\mathbb{R})$ satisfy, for some $\rho\geq0$, \[\frac{x}{2}\phi'(x)+G(\phi''(x))=\rho \phi(x).\] It is easy to check that
$u(t,x)=(1+t)^\rho \phi(\frac{x}{\sqrt{1+t}})$ is a solution to the $G$-heat equation. Therefore
\begin {eqnarray}\label {exam2}u(t,x)=\mathcal{N}_G[\phi(x+\sqrt{t}\cdot)]=\mu^{t,x}[\phi(x+\sqrt{t}\cdot)]=(1+t)^\rho \phi(\frac{x}{\sqrt{1+t}}),
 \end {eqnarray} where $\mu^{t,x}$ is a realization of $\phi(x+\sqrt{t}\cdot)$.  By taking formal  derivation on  (\ref {exam2}) with respect to $t$, we get \begin {eqnarray} \label {exam2-1}
  \partial_t u(t,x)&=&\int_\mathbb{R}\phi'(x+\sqrt{t}y)\frac{y}{2\sqrt{t}}\mu^{t,x}(dy)\\
 \label {exam2-2} &=&\rho(1+t)^{\rho-1} \phi(\frac{x}{\sqrt{1+t}})-\frac{x}{2}(1+t)^{\rho-\frac{3}{2}}\phi'(\frac{x}{\sqrt{1+t}}).
  \end {eqnarray} Similarly, by taking formal  derivation on  (\ref {exam2}) with respect to $x$, we get \begin {eqnarray} \label {exam2-3}
  \partial_x u(t,x)=\int_\mathbb{R}\phi'(x+\sqrt{t}y)\mu^{t,x}(dy)=(1+t)^{\rho-\frac{1}{2}}\phi'(\frac{x}{\sqrt{1+t}}).
  \end {eqnarray} Note that $(\ref {exam2-2})\times (1+t)+ (\ref {exam2-3})\times \frac{x}{2}-(\ref {exam2})\times\rho=0,$ which implies
  \[\int_\mathbb{R}[\frac{y}{2\sqrt{t}}\phi'(x+\sqrt{t}y)-G(\phi''(x+\sqrt{t}y))]\mu^{t,x}(dy)=0.\] More precisely, we have
  \[\int_\mathbb{R}[\frac{y}{2}\partial_y\phi(x+\sqrt{t}y)-G(\partial_y^2\phi(x+\sqrt{t}y))]\mu^{t,x}(dy)=0,\] which is just the conclusion of Proposition \ref {main}.
\end {example}
Return to the linear case, the closed linear span  of the family of functions considered in either of the above two examples is the space of continuous functions, which  makes us confident that the conclusion of Proposition \ref {main} is correct.

Just like Stein's characterization of (classical) normal distributions, we are also concerned about the converse problem:

\noindent (\textbf{Q})\emph{ Let $\mathcal{N}[\varphi]=\max_{\mu\in\Theta}\mu[\varphi],\ \varphi\in C_{b,Lip}(\mathbb{R}),$ be a sublinear expectation. Assuming   $\mathcal{N}$ satisfies the  Stein type formula (SH) below, does it follow that $\mathcal{N}=\mathcal{N}_G $?}

 (\textbf{SH}) \   For $\varphi\in C_b^2(\mathbb{R})$, we have
 \[\int_\mathbb{R}[G(\varphi''(x))-\frac{x}{2}\varphi'(x)]\mu^\varphi(dx)=0,\] where $\mu^\varphi$ is a realization of $\varphi$ associated with $\mathcal{N}$, i.e., $\mu^\varphi\in \Theta$ and $\mu^\varphi[\varphi]=\mathcal{N}[\varphi]$.

Actually, we can also find evidences for the converse statement from some simple examples.
\begin {example} Assume that $\mathcal{N}$ is a sublinear expectation on $C_{b,Lip}(\mathbb{R})$  satisfying the  Stein type formula (SH). Set $u(t,x):=\mathcal{N}[\phi_\beta(x+\sqrt{t}\cdot)]$. We shall ``prove" that $u$ is the solution to the $G$-heat equation.  Actually, noting that $u(t,x)=\mathcal{N}[\phi_\beta(x+\sqrt{t}\cdot)]=\mu^{t,x}[\phi_\beta(x+\sqrt{t}\cdot)]$
with $\mu^{t,x}$  a realization of $\phi_\beta(x+\sqrt{t}\cdot)$, we get  \[\partial_tu(t,x)=\int_\mathbb{R}\phi'_\beta(x+\sqrt{t}y)\frac{y}{2\sqrt{t}}\mu^{t,x}(dy).\] So, from Hypothesis (SH), we get
\[\partial_tu(t,x)=\int_\mathbb{R}G(\phi''_\beta(x+\sqrt{t}y))\mu^{t,x}(dy)=-\frac{\sigma^2}{2}\int_\mathbb{R}\phi_\beta(x+\sqrt{t}y)\mu^{t,x}(dy)=-\frac{\sigma^2}{2}u(t,x).\] Then  $u(t,x)=e^{-\frac{\sigma^2}{2}t}\phi_\beta(x),$ which is the solution to the $G$-heat equation with  $u(0,x)=\phi_\beta(x)$.
\end {example}
Our purpose is to prove the Stein type formula for $G$-normal distributions (Proposition \ref {main}) and its converse problem (Q). In order to do so, we first prove a weaker version of the Stein type characterization below.
\begin {theorem}\label {Stein-C}
 Let $\mathcal{N}[\varphi]=\max_{\mu\in\Theta}\mu[\varphi],\ \varphi\in C_{b,Lip}(\mathbb{R}),$ be a sublinear expectation. $\mathcal{N}$ is $G$-normal if and only if  for any $\varphi\in C_b^2(\mathbb{R})$, we have
 \begin {align}\label {SHw}\sup_{\mu\in\Theta_\varphi}\int_\mathbb{R}[G(\varphi''(x))-\frac{x}{2}\varphi'(x)]\mu(dx)=0,\tag{SHw}
 \end {align} where $\Theta_\varphi=\{\mu\in \Theta:\mu[\varphi]=\mathcal{N}[\varphi]\}.$
\end {theorem}
Since Hypothesis (\ref{SHw}) is weaker than  (\ref{SH}), the necessity of Theorem \ref {Stein-C} is weaker than that of  Proposition \ref {main}. At the same time, the sufficiency of Theorem \ref {Stein-C} implies the converse arguments (Q).

In Section 2 we provide two lemmas to show how the differentiation penetrates the sublinear expectations, which makes sense the ``formal derivation" in the above examples. In section 3, we give a proof to Theorem \ref {Stein-C}. We shall prove Proposition \ref {main} in Section 5 based on the $G$-expectation theory, and as a preparation we list some basic definitions and notations concerning $G$-expectation in Section 4.

\section {Two Useful Lemmas}
Let $\mathcal{N}[\varphi]=\max_{\mu\in\Theta}\mu[\varphi]$ be a sublinear expectation on  $C_{b,Lip}(\mathbb{R})$. Throughout this article, we suppose the following additional properties:
\begin{description}
\item[(H1)] $\Theta$ is weakly compact;

\item[(H2)] $\lim_{N\rightarrow\infty}\mathcal{N}[|x|1_{[|x|>N]}]=0.$
\end{description}
 Clearly, the tightness of $\Theta$ is already implied by (H2). We  emphasize by (H1) that $\Theta$ is weakly closed, which ensures that there exists a realization $\mu_\varphi$ for any $\varphi\in C_{b,Lip}(\mathbb{R})$.
In addition, it is easily seen that $\Theta$ can always be chosen as weakly closed for a sublinear expectations $\mathcal{N}$ on  $C_{b,Lip}(\mathbb{R})$.

Define $\xi: \mathbb{R}\rightarrow\mathbb{R}$ by $\xi(x)=x$. Sometimes, we write $\mathcal{N}[\varphi], \ \mu[\varphi]$ by $\mathbb{E}[\varphi(\xi)], \ E_\mu[\varphi(\xi)]$, respectively. For  $\varphi\in C_{b,Lip}(\mathbb{R})$, set $\Theta_\varphi=\{\mu\in \Theta:E_{\mu}[\varphi(\xi)]=\mathbb{E}[\varphi(\xi)]\}$.

For a sublinear expectation $\mathcal{N}$ on $C_{b,Lip}(\mathbb{R})$
and $\varphi\in C_{b,Lip}(\mathbb{R})$, set \[u(t,x):=\mathcal{N}[\varphi(x+\sqrt{t}\cdot)]=\mathbb{E}[\varphi(x+\sqrt{t}\xi)].\]

Define $\partial_{t}^+u(t,x)=\lim_{\delta\downarrow0}\frac{u(t+\delta,x)-u(t,x)}{\delta}$ (respectively, $\partial_{t}^-u(t,x)=\lim_{\delta\downarrow0}\frac{u(t-\delta,x)-u(t,x)}{-\delta})$ ) if the corresponding limits exist. Similarly, we can define  $\partial_{x}^+u(t,x)$ and  $\partial_{x}^-u(t,x)$.

\begin {lemma}\label {lem-deri} For a sublinear expectation $\mathcal{N}$ on $C_{b,Lip}(\mathbb{R})$
and $\varphi\in C_b^2(\mathbb{R})$, we have
\begin {eqnarray} \label {lem-deri-e1} \partial_{t}^+u(t,x)=\sup_{\mu\in\Theta_{t,x}}E_\mu[\frac{\xi}{2\sqrt{t}}\varphi'(x+\sqrt{t}\xi)], & & \partial_{t}^-u(t,x)=\inf_{\mu\in\Theta_{t,x}}E_\mu[\frac{\xi}{2\sqrt{t}}\varphi'(x+\sqrt{t}\xi)],\\
 \label {lem-deri-e2}\partial_{x}^+u(t,x)=\sup_{\mu\in\Theta_{t,x}}E_\mu[\varphi'(x+\sqrt{t}\xi)], & & \partial_{x}^-u(t,x)=\inf_{\mu\in\Theta_{t,x}}E_\mu[\varphi'(x+\sqrt{t}\xi)],
\end {eqnarray} where $\Theta_{t,x}=\Theta_{\varphi(x+\sqrt{t}\cdot)}$. Furthermore, we have
\begin {eqnarray}  \label {lem-deri-e3} \partial_{t}^+u(t,x)=\lim_{\delta\downarrow0}\partial_{t}^+u(t+\delta,x)=\lim_{\delta\downarrow0}\partial_{t}^-u(t+\delta,x), \\  \label {lem-deri-e4} \partial_{t}^-u(t,x)=\lim_{\delta\downarrow0}\partial_{t}^+u(t-\delta,x)=\lim_{\delta\downarrow0}\partial_{t}^-u(t-\delta,x).
\end {eqnarray} Similar relations hold for $\partial_{x}^+u(t,x), \partial_{x}^-u(t,x)$.

\end {lemma}\label {lem-deri2}
\begin {proof} We shall only give proof to (\ref {lem-deri-e1}) and  (\ref {lem-deri-e3}). The other conclusions can be proved similarly.

\textbf{Step 1.} Proof to (\ref {lem-deri-e1}).

By the definition of the function $u$ we have, for any $\mu^\delta\in \Theta_{t+\delta,x}$,
\begin{eqnarray}
\label {lem-deri-proof-e0}\frac{u(t+\delta,x)-u(t,x)}{\delta}&=&\frac{1}{\delta}\mathbb{E}[\varphi(x+\sqrt{t+\delta}\xi)]-\frac{1}{\delta}\mathbb{E}[\varphi(x+\sqrt{t}\xi)]\\
\label {lem-deri-proof-e1} &=& \frac{1}{\delta}E_{\mu^\delta}[\varphi(x+\sqrt{t+\delta}\xi)]-\frac{1}{\delta}\mathbb{E}[\varphi(x+\sqrt{t}\xi)]\\
\label {lem-deri-proof-e2}&\leq& \frac{1}{\delta}E_{\mu^\delta}[\varphi(x+\sqrt{t+\delta}\xi)]-\frac{1}{\delta}E_{\mu^\delta}[\varphi(x+\sqrt{t}\xi)]\\
\label {lem-deri-proof-e3} &=&E_{\mu^{\delta}}[\frac{\xi}{2\sqrt{t}}\varphi^{\prime}(x+\sqrt{t}\xi)]+o(1).
\end{eqnarray}

There exists a subsequence (still denoted by $\mu^{\delta}$) such that
$\mu^{\delta}\xrightarrow
{weakly} \mu^{*}\in \Theta$.

 Then, by (\ref {lem-deri-proof-e3}), we have
\[
\limsup_{\delta \downarrow0}\frac{u(t+\delta,x)-u(t,x)}{\delta}\leq \lim_{\delta \downarrow0} E_{\mu^{\delta}}[\frac{\xi}{2\sqrt{t}}\varphi^{\prime}(x+\sqrt{t}\xi)]=
E_{\mu^{*}}[\frac{\xi}{2\sqrt{t}}\varphi^{\prime}(x+\sqrt{t}\xi)].\]
Noting that
\begin {eqnarray*}|E_{\mu^\delta}[\varphi(x+\sqrt{t+\delta}\xi)]-E_{\mu^\delta}[\varphi(x+\sqrt{t}\xi)]|\leq \mathbb{E}[|\varphi(x+\sqrt{t+\delta}\xi)-\varphi(x+\sqrt{t}\xi)|]\rightarrow0,
\end {eqnarray*} we conclude, from (\ref {lem-deri-proof-e1}), that
\[E_{\mu^*}[\varphi(x+\sqrt{t}\xi)]=\lim_{\delta\downarrow 0}E_{\mu^\delta}[\varphi(x+\sqrt{t+\delta}\xi)]=\mathbb{E}[\varphi(x+\sqrt{t}\xi)],\]
 which means that $\mu^*\in \Theta_{t,x}$.

On the other hand, for any $\mu\in \Theta_{t,x}$ we get
\begin{eqnarray}
\label {lem-deri-proof-e4}\frac{u(t+\delta,x)-u(t,x)}{\delta} &=&\frac{1}{\delta
}\mathbb{E}[\varphi(x+\sqrt{t+\delta}\xi)]-\frac{1}{\delta
}E_\mu[\varphi(x+\sqrt{t}\xi)]\\
\label {lem-deri-proof-e5}&\geq&\frac{1}{\delta
}E_\mu[\varphi(x+\sqrt{t+\delta}\xi)]-\frac{1}{\delta
}E_\mu[\varphi(x+\sqrt{t}\xi)]\\
\label {lem-deri-proof-e6}&=&E_\mu[\frac{\xi}{2\sqrt{t}}\varphi^{\prime}(x+\sqrt{t}\xi)]+o(1).
\end{eqnarray}

Thus, by (\ref{lem-deri-proof-e6}), we get
\[
\liminf_{\delta \downarrow0}\frac{u(t+\delta,x)-u(t,x)}{\delta}\geq\sup_{\mu\in \Theta_{t,x}}
E_{\mu}[\frac{\xi}{2\sqrt{t}}\varphi^{\prime}(x+\sqrt{t}\xi)].
\] Combining the above arguments, we have \[
\lim_{\delta \downarrow0}\frac{u(t+\delta,x)-u(t,x)}{\delta}=\sup_{\mu\in \Theta_{t,x}}%
E_{\mu}[\frac{\xi}{2\sqrt{t}}\varphi^{\prime}(x+\sqrt{t}\xi)]=E_{\mu^*}[\frac{\xi}{2\sqrt{t}}\varphi^{\prime}(x+\sqrt{t}\xi)].
\]

\textbf{Step 2.} Proof to (\ref {lem-deri-e4}).

By (\ref {lem-deri-e1}), there exists $\mu^\delta\in \Theta_{t+\delta, x}$ such that $\partial_{t}^+u(t+\delta,x)=E_{\mu^\delta}[\frac{\xi}{2\sqrt{t+\delta}}\varphi'(x+\sqrt{t+\delta}\xi)]$. Noting that
\[E_{\mu^\delta}[\frac{\xi}{2\sqrt{t+\delta}}\varphi'(x+\sqrt{t+\delta}\xi)]-E_{\mu^\delta}[\frac{\xi}{2\sqrt{t}}\varphi'(x+\sqrt{t}\xi)]\rightarrow0,\]
it suffices to prove that $E_{\mu^\delta}[\frac{\xi}{2\sqrt{t}}\varphi'(x+\sqrt{t}\xi)]\rightarrow \partial_{t}^+u(t,x)$. Actually, for any subsequence of $(\mu^\delta)$, there is a sub-subsequence ($\mu^{\delta'}$) such that $\mu^{\delta'}\xrightarrow
{weakly} \mu^{*}\in \Theta$. By Step 1 we have $\mu^{*}\in \Theta_{t,x}$ and $\partial_{t}^+u(t,x)=E_{\mu^*}[\frac{\xi}{2\sqrt{t}}\varphi^{\prime}(x+\sqrt{t}\xi)]$. So we conclude that $E_{\mu^{\delta'}}[\frac{\xi}{2\sqrt{t}}\varphi'(x+\sqrt{t}\xi)]\rightarrow \partial_{t}^+u(t,x)$.
\end {proof}

For any $\varphi \in C_{b,Lip}(\mathbb{R})$, let $v(t,x)$ be the solution to
the $G$-heat equation with initial value $\varphi$. For a sublinear expectation $\mathcal{N}$ on $C_{b,Lip}(\mathbb{R})$,  set
$w_{\mathcal{N}}(s):=\mathbb{E}[v(s,\sqrt{1-s}\xi)]$,  $s\in \lbrack0,1]$.
\begin {lemma}\label {lem-deri2} For a sublinear expectation $\mathcal{N}$ on $C_{b,Lip}(\mathbb{R})$, we have, for $t\in(0,1)$
\begin {eqnarray}\label {lem-deri2-e1}\partial_t^+w_{\mathcal{N}}(t)=\sup_{\mu\in\Theta_{v(t,\sqrt{1-t}\cdot)}}E_\mu[\partial_{t}v(t,\sqrt{1-t}\xi)-\partial_{x}v(t,\sqrt{1-t}\xi)\frac{\xi
}{2\sqrt{1-t}}],  \\
\label {lem-deri2-e2}\partial_t^-w_{\mathcal{N}}(t)=\inf_{\mu\in\Theta_{v(t,\sqrt{1-t}\cdot)}}E_\mu[\partial_{t}v(t,\sqrt{1-t}\xi)-\partial_{x}v(t,\sqrt{1-t}\xi)\frac{\xi
}{2\sqrt{1-t}}].
\end {eqnarray}
\end {lemma}
\begin {proof}For any $\mu^\delta \in \Theta_{v(t+\delta,\sqrt{1-t-\delta}\cdot)}$, we have
\begin{eqnarray}
\label {lem-deri2-proof-e0} w_{\mathcal{N}}(t+\delta)-w_{\mathcal{N}}(t)&=&\mathbb{E}[v(t+\delta,\sqrt{1-t-\delta}\xi)]-\mathbb{E}%
[v(t,\sqrt{1-t}\xi)]\\
\label {lem-deri2-proof-e1}&=&E_{\mu^{\delta}}[v(t+\delta,\sqrt{1-t-\delta}\xi)]-\mathbb{E}[v(t,\sqrt
{1-t}\xi)]\\
\label {lem-deri2-proof-e2}&\leq& E_{\mu^{\delta}}[v(t+\delta,\sqrt{1-t-\delta}\xi)]-E_{\mu^{\delta}%
}[v(t,\sqrt{1-t}\xi)]\\
\label {lem-deri2-proof-e3}&=&E_{\mu^{\delta}}[\partial_{t}v(t,\sqrt{1-t}\xi)\delta-\partial_{x}%
v(t,\sqrt{1-t}\xi)\frac{\xi}{2\sqrt{1-t}}\delta]+o(\delta).
\end{eqnarray}
There exists a subsequence (still denoted by $\mu^{\delta}$) such that
$\mu^{\delta}\xrightarrow{weakly}\mu^{*}\in \Theta$. Noting that
\begin {eqnarray*}
& &|E_{\mu^{\delta}}[v(t+\delta,\sqrt{1-t-\delta}\xi)]-E_{\mu^{\delta}}[v(t,\sqrt{1-t}\xi)]|\\
&\leq& \mathbb{E}[|v(t+\delta,\sqrt{1-t-\delta}\xi)-v(t,\sqrt{1-t}\xi)|]\rightarrow 0,
\end {eqnarray*} we have, by (\ref{lem-deri2-proof-e1}),
\[E_{\mu^*}[v(t,\sqrt{1-t}\xi)]=\lim_{\delta\downarrow0}E_{\mu^{\delta}}[v(t+\delta,\sqrt{1-t-\delta}\xi)]=\mathbb{E}[v(t,\sqrt
{1-t}\xi)],\] i.e.,  $\mu^{*}$
belongs to $\Theta_{v(t,\sqrt{1-t}\cdot)}$. Then   by (\ref{lem-deri2-proof-e3}) we have
\[
\limsup_{\delta \downarrow0}\frac{w_{\mathcal{N}}(t+\delta)-w_{\mathcal{N}}(t)}{\delta}\leq E_{\mu^{*}%
}[\partial_{t}v(t,\sqrt{1-t}\xi)-\partial_{x}v(t,\sqrt{1-t}\xi)\frac{\xi
}{2\sqrt{1-t}}].
\]
On the other hand,  for each fixed $\mu\in \Theta_{v(t,\sqrt{1-t}\cdot
)}$, we have%
\begin{align}
\label {lem-deri2-proof-e4}w_{\mathcal{N}}(t+\delta)-w_{\mathcal{N}}(t)  & =\mathbb{E}[v(t+\delta,\sqrt{1-t-\delta}\xi)]-\mathbb{E}%
[v(t,\sqrt{1-t}\xi)]\\
\label {lem-deri2-proof-e5}& =\mathbb{E}[v(t+\delta,\sqrt{1-t-\delta}\xi)]-E_{\mu}[v(t,\sqrt{1-t}\xi)]\\
\label {lem-deri2-proof-e6}& \geq E_{\mu}[v(t+\delta,\sqrt{1-t-\delta}\xi)]-E_{\mu}[v(t,\sqrt{1-t}\xi)]\\
\label {lem-deri2-proof-e7}& =E_{\mu}[\partial_{t}v(t,\sqrt{1-t}\xi)\delta-\partial_{x}v(t,\sqrt{1-t}%
\xi)\frac{\xi}{2\sqrt{1-t}}\delta]+o(\delta).
\end{align}

Therefore, we conclude, by (\ref {lem-deri2-proof-e7}), that \[\liminf_{\delta \downarrow0}\frac{w_{\mathcal{N}}(t+\delta
)-w_{\mathcal{N}}(t)}{\delta}\geq \sup_{\mu\in\Theta_{v(t,\sqrt{1-t}\cdot)}}E_\mu[\partial_{t}v(t,\sqrt{1-t}\xi)-\partial_{x}v(t,\sqrt{1-t}\xi)\frac{\xi
}{2\sqrt{1-t}}].\] Combining the above arguments, we obtain (\ref{lem-deri2-e1}). By similar arguments, we can prove (\ref{lem-deri2-e2}).

\end {proof}
\section {Proof to Theorem \ref {Stein-C}}
We shall prove Theorem \ref {Stein-C} based mainly on the two lemmas introduced in Section 2.

\begin {proof} \textbf{Necessity.}

Assume that $\mathcal{N}$ is $G$-normal. Then, for $\varphi\in C_b^2(\mathbb{R})$, $u(t,x):=\mathcal{N}[\varphi(x+\sqrt{t}\cdot)]=\mathbb{E}[\varphi(x+\sqrt{t}\xi)]$ is the solution to
$G$-heat equation with initial value $\varphi$.

\textbf{Step 1.} For $\mu\in \Theta_\varphi$, \[\partial_tu(1,0)=E_{\mu}[\frac{\xi}{2}\varphi^{\prime}(\xi)].\] Actually, by Lemma \ref {lem-deri}, we have
\[\sup_{\mu\in\Theta_\varphi}E_\mu[\frac{\xi}{2}\varphi'(\xi)]=\partial_t^+u(1,0)=\partial_tu(1,0)=\partial_t^-u(1,0)=\inf_{\mu\in\Theta_\varphi}E_\mu[\frac{\xi}{2}\varphi'(\xi)].\]

\textbf{Step 2.} $\partial_tu(1,0)=\sup_{\mu\in \Theta_\varphi}E_{\mu}[G(\varphi^{''}(\xi))].$

Note that $u(1+\delta, 0)=\mathbb{E}[u(\delta, \xi)]$ and $u(\delta,x)=\mathbb{E}[\varphi(x+\sqrt{\delta} \xi)]=\varphi(x)+\delta G(\varphi''(x))+o(\delta)$ uniformly with respect to $x$. So, for $\mu^\delta\in\Theta_{u(\delta, \cdot)}$, we have
\begin {eqnarray}
\label {thm-main-proof-e0} u(1+\delta, 0)-u(1, 0)&=&\mathbb{E}[u(\delta, \xi)]-\mathbb{E}[\varphi( \xi)]\\
\label {thm-main-proof-e1} &=& E_{\mu^\delta}[u(\delta, \xi)]-\mathbb{E}[\varphi(\xi)]\\
\label {thm-main-proof-e2} &\leq& E_{\mu^\delta}[u(\delta, \xi)]-E_{\mu^\delta}[\varphi(\xi)]\\
\label {thm-main-proof-e3} &=& \delta E_{\mu^\delta}[G(\varphi''(\xi))] +o(\delta).
\end {eqnarray}
There exists a subsequence (still denoted by $\mu^{\delta}$) such that
$\mu^{\delta}\xrightarrow{weakly}\mu^{*}\in \Theta$. Noting that
\begin {eqnarray*}
|E_{\mu^{\delta}}[u(\delta, \xi)]-E_{\mu^{\delta}}[\varphi(\xi)]|\leq \mathbb{E}[|u(\delta, \xi)-\varphi(\xi)|]\rightarrow 0,
\end {eqnarray*} we have, by (\ref{thm-main-proof-e1}),
\[E_{\mu^*}[\varphi(\xi)]=\lim_{\delta\downarrow0}E_{\mu^{\delta}}[u(\delta, \xi)]=\mathbb{E}[\varphi(\xi)],\] i.e.,  $\mu^{*}$
belongs to $\Theta_{\varphi}$.

By (\ref {thm-main-proof-e3}) we obtain
\[\limsup_{\delta\downarrow0}\frac{u(1+\delta,0)-u(1,0)}{\delta}\leq E_{\mu^{*}}[G(\varphi''(\xi))].\]
On the other hand, for any $\mu\in \Theta_\varphi$, we have
\begin {eqnarray}
\label {thm-main-proof-e4} u(1+\delta, 0)-u(1, 0)&=&\mathbb{E}[u(\delta, \xi)]-\mathbb{E}[\varphi( \xi)]\\
\label {thm-main-proof-e5} &=& \mathbb{E}[u(\delta, \xi)]-E_\mu[\varphi(\xi)]\\
\label {thm-main-proof-e6} &\geq& E_{\mu}[u(\delta, \xi)]-E_{\mu}[\varphi(\xi)]\\
\label {thm-main-proof-e7} &=& \delta E_{\mu}[G(\varphi''(\xi))] +o(\delta).
\end {eqnarray} By (\ref{thm-main-proof-e7}), we get $\liminf_{\delta\downarrow0}\frac{u(1+\delta,0)-u(1,0)}{\delta}\geq\sup_{\mu\in \Theta_\varphi}E_{\mu}[G(\varphi^{''}(\xi))].$ Combining the above arguments, we get the desired result.

\textbf{Sufficiency.}

Assume $\mathcal{N}$ is a sublinear expectation on $C_{b,Lip}(\mathbb{R})$ satisfying Hypothesis (\ref {SHw}). For any $\varphi \in C_{b,Lip}(\mathbb{R})$, let $v(t,x)$ be the solution to
the $G$-heat equation with initial value $\varphi$. For $s\in \lbrack0,1]$, set
$w(s):=\mathbb{E}[v(s,\sqrt{1-s}\xi)]$. To prove the theorem, it suffices to
show that $w(0)=w(1)$.

By (\ref{lem-deri2-e1} ) in Lemma \ref {lem-deri2} and Hypothesis (\ref{SHw}), we get $\partial_s^+w(s)=0, s\in (0,1)$. Noting that $w$ is continuous on $[0,1]$ and locally Lipschitz continuous on $(0,1)$, we get $w(0)=w(1)$.

\end {proof}

\begin {corollary} Let $\mathcal{N}[\varphi]=\max_{\mu\in\Theta}\mu[\varphi],\ \varphi\in C_{b,Lip}(\mathbb{R}),$ be a sublinear expectation. Then $\mathcal{N}$ is $G$-normal if  for any $\varphi\in C_b^2(\mathbb{R})$, we have
 \begin {align}\label {SH}\int_\mathbb{R}[G(\varphi''(x))-\frac{x}{2}\varphi'(x)]\mu^\varphi(dx)=0,\tag{SH}
 \end {align} where $\mu^\varphi$ is  a  realization of $\varphi$ associated with $\mathcal{N}$, i.e., $\mu^\varphi\in \Theta$ and $\mu^\varphi[\varphi]=\mathcal{N}[\varphi]$.
\end {corollary}

\section{Some Definitions and Notations about $G$-expectation}
We review some basic notions and definitions of the related spaces under $G$%
-expectation. The readers may refer to \cite{P07a}, \cite{P07b}, \cite{P08a}%
, \cite{P08b}, \cite{P10} for more details.

Let $\Omega_T=C_{0}([0,T];\mathbb{R}^{d})$ be the space of all $%
\mathbb{R}^{d}$-valued continuous paths $\omega=(\omega(t))_{t\geq0}\in
\Omega$ with $\omega(0)=0$ and let $B_{t}(\omega)=\omega(t)$ be the
canonical process.

Let us recall the definitions of $G$-Brownian motion and its corresponding $%
G $-expectation introduced in \cite{P07b}. Set
\begin{equation*}
L_{ip}(\Omega_{T}):=\{
\varphi(\omega(t_{1}),\cdots,\omega(t_{n})):t_{1},\cdots,t_{n}\in
\lbrack0,T],\  \varphi \in C_{b,Lip}((\mathbb{R}^{d})^{n}),\ n\in \mathbb{N}%
\},
\end{equation*}
where $C_{b,Lip}(\mathbb{R}^{d})$ is the collection of bounded Lipschitz
functions on $\mathbb{R}^{d}$.

We are given a function%
\begin{equation*}
G:\mathbb{S}_d\mapsto \mathbb{R}
\end{equation*}
satisfying the following monotonicity,  sublinearity and positive homogeneity:

\begin{description}
\item[A1.] \bigskip$G(a)\geq G(b),\  \ $if $a,b\in \mathbb{S}_d$
and $a\geq b;$

\item[A2.] $G(a+b)\leq G(a)+G(b)$, for each $a,b\in \mathbb{S}_d$;

\item[A3.] $  G(\lambda a)=\lambda
G(a)$ for  $a\in \mathbb{S}_d$ and $\lambda \geq0.$
\end{description}

\begin{remark}
When $d=1$, we have $G(a):=\frac{1}{2}(\overline{\sigma}^{2}a^{+}-\underline{%
\sigma}^{2}a^{-})$, for $0\leq \underline{\sigma}^{2}\leq \overline{\sigma}%
^{2}$.
\end{remark}

\bigskip \ For each $\xi(\omega)\in L_{ip}(\Omega_{T})$ of the form
\begin{equation*}
\xi(\omega)=\varphi(\omega(t_{1}),\omega(t_{2}),\cdots,\omega(t_{n})),\  \
0=t_{0}<t_{1}<\cdots<t_{n}=T,
\end{equation*}
we define the following conditional  $G$-expectation
\begin{equation*}
\mathbb{E}_{t}[\xi]:=u_{k}(t,\omega(t);\omega(t_{1}),\cdots,\omega
(t_{k-1}))
\end{equation*}
for each $t\in \lbrack t_{k-1},t_{k})$, $k=1,\cdots,n$. Here, for each $%
k=1,\cdots,n$, $u_{k}=u_{k}(t,x;x_{1},\cdots,x_{k-1})$ is a function of $%
(t,x)$ parameterized by $(x_{1},\cdots,x_{k-1})\in (\mathbb{R}^d)^{k-1}$, which
is the solution of the following PDE ($G$-heat equation) defined on $%
[t_{k-1},t_{k})\times \mathbb{R}^d$:
\begin{equation*}
\partial_{t}u_{k}+G(\partial^2_{x}u_{k})=0\
\end{equation*}
with terminal conditions
\begin{equation*}
u_{k}(t_{k},x;x_{1},\cdots,x_{k-1})=u_{k+1}(t_{k},x;x_{1},\cdots x_{k-1},x),
\, \, \hbox{for $k<n$}
\end{equation*}
and $u_{n}(t_{n},x;x_{1},\cdots,x_{n-1})=\varphi (x_{1},\cdots x_{n-1},x)$.

The $G$-expectation of $\xi(\omega)$ is defined by $\mathbb{E}[\xi]=%
\mathbb{E}_{0}[\xi]$. From this construction we obtain a natural norm $%
\left \Vert \xi \right \Vert _{L_{G}^{p}}:=\mathbb{E}[|\xi|^{p}]^{1/p}$, $p\geq 1$.
The completion of $L_{ip}(\Omega_{T})$ under $\left \Vert \cdot \right \Vert
_{L_{G}^{p}}$ is a Banach space, denoted by $L_{G}^{p}(\Omega_{T})$. The
canonical process $B_{t}(\omega):=\omega(t)$, $t\geq0$, is called a $G$%
-Brownian motion in this sublinear expectation space $(\Omega,L_{G}^{1}(%
\Omega ),\mathbb{E})$.

\begin {definition} A process $\{M_t\}$ with values in
$L^1_G(\Omega_T)$ is called a $G$-martingale if $\mathbb{E}_s(M_t)=M_s$
for any $s\leq t$. If $\{M_t\}$ and  $\{-M_t\}$ are both
$G$-martingales, we call $\{M_t\}$ a symmetric $G$-martingale.
\end {definition}

\begin{theorem}
\label{the2.7} (\cite{DHP11}) There exists a weakly compact subset $\mathcal{P}%
\subset\mathcal{M}_{1}(\Omega_{T})$, the set of probability measures on
$(\Omega_{T},\mathcal{B}(\Omega_{T}))$, such that
\[
\mathbb{E}[\xi]=\sup_{P\in\mathcal{P}}E_{P}[\xi]\ \ \text{for
\ all}\ \xi\in L_{ip}(\Omega_T).
\]
$\mathcal{P}$ is called a set that represents $\mathbb{E}$.
\end{theorem}

\begin{definition}
A function $\eta (t,\omega ):[0,T]\times \Omega _{T}\rightarrow \mathbb{R}$
is called a step process if there exists a time partition $%
\{t_{i}\}_{i=0}^{n}$ with $0=t_{0}<t_{1}<\cdot \cdot \cdot <t_{n}=T$, such
that for each $k=0,1,\cdot \cdot \cdot ,n-1$ and $t\in (t_{k},t_{k+1}]$
\begin{equation*}
\eta (t,\omega )=\xi_{t_k}\in L_{ip}(\Omega_{t_k}).
\end{equation*}%
 We denote by $M^{0}(0,T)$ the
collection of all step processes.
\end{definition}

 For a step process $\eta \in
M^{0}(0,T) $, we set the norm
\begin{equation*}
\Vert \eta \Vert _{H_{G}^{p}}^{p}:=\mathbb{E}[\{ \int_{0}^{T}|\eta
_{s}|^{2}ds\}^{p/2}], p\geq 1
\end{equation*}
and denote by $H_{G}^{p}(0,T)$ the completion of $%
M^{0}(0,T)$ with respect to the norms $\Vert \cdot \Vert _{H_{G}^{p}}$.

\begin{theorem}
(\cite{Song11}) For $\xi\in
L^\beta_G(\Omega_T)$ with some $\beta>1$, $X_t=\mathbb{E}_t(\xi)$, $
t\in[0, T]$ has the following decomposition:
\begin {eqnarray*}
X_t=X_0+\int_0^tZ_sdB_s+K_t, \ q.s.,
\end {eqnarray*}
 where $\{Z_t\}\in H^1_G(0, T)$  and $\{K_t\}$ is a continuous
 non-increasing $G$-martingale. Furthermore, the above decomposition is unique and
$\{Z_t\}\in H^\alpha_G(0, T)$, $K_T\in L^\alpha_G(\Omega_T)$ for any
$1\leq\alpha<\beta$.
\end{theorem}

\section{Proof to Proposition \ref {main}}

 Let $\mathcal{P}$ be a  weakly compact set that represents $\mathbb{E}$. Then the corresponding $G$-normal distribution can be represented as
 \[\mathcal{N}_G[\varphi]=\max_{P\in \mathcal{P}}E_P[\varphi(B_1)], \ \textmd{for all} \  \varphi\in C_{b,Lip}(\mathbb{R}).\]  Clearly, $\mathcal{N}_G$ satisfies condition (H2) and $\Theta:=\{P\circ B_1^{-1}| \ P\in \mathcal{P}\}$ is weakly compact. Also, Proposition \ref {main} can be restated in the following form.

 \begin {proposition} Let $\varphi\in C^2_b(\mathbb{R})$.  For $P\in \mathcal{P}$ such that $E_{P}[\varphi(B_1)=\mathbb{E}[\varphi(B_1)]$, we have
 \[E_{P}[\frac{B_{1}}{2}\varphi^{\prime}(B_{1})-G(\varphi^{\prime \prime}(B_{1}))]=0.\]
\end {proposition}

\begin {proof} For $\varphi\in C^2_b(\mathbb{R})$, set $u(t,x)=\mathbb{E}[\varphi(x+B_{t})]$.  As a solution to the $G$-heat equation, we know $u\in C^{1,2}_b(\mathbb{R})$. Particularly, one has
\[
\lim_{\delta \downarrow0}\frac{u(1+\delta,0)-u(1,0)}{\delta}=\lim
_{\delta \downarrow0}\frac{u(1-\delta,0)-u(1,0)}{-\delta}=\partial_tu(1,0).
\]
Set $\mathcal{P}_\varphi=\{P\in \mathcal{P}:E_{P}[\varphi(B_1)=\mathbb{E}[\varphi(B_1)]\}.$ In the proof to Theorem \ref {Stein-C}, we have already  proved that
 for $P\in \mathcal{P}_\varphi$, $\partial_tu(1,0)=E_{P}[\frac{B_{1}}{2}\varphi^{\prime}(B_{1})]$ and that $\partial_tu(1,0)=\sup_{P\in \mathcal{P}_\varphi}
E_{P}[G(\varphi^{\prime \prime}(B_{1}))].$ We shall just  prove $
\partial_tu(1,0)=\inf_{P\in \mathcal{P}_\varphi}%
E_{P}[G(\varphi^{\prime \prime}(B_{1}))].$

By the $G$-martingale representation theorem,  we have
\[%
\begin{array}
[c]{l}%
u(1-\delta,0)=\varphi(B_{1-\delta})-\int_{0}^{1-\delta}z_{s}^{\delta}%
dB_{s}-K_{1-\delta}^{\delta},\\
u(1,0)=\varphi(B_{1})-\int_{0}^{1}z_{s}dB_{s}-K_{1},
\end{array}
\] where $\{K^{\delta}_t\}, \{K_t\}$ are non-increasing $G$-martingales with $K^{\delta}_0=K_0=0$. Thus%
\[%
\begin{array}
[c]{rl}%
u(1-\delta,0)-u(1,0) & =\mathbb{E}[\varphi(B_{1-\delta})-\varphi
(B_{1})+K_{1}]\\
& =\mathbb{E}[-\frac{1}{2}\varphi^{\prime \prime}(B_{1-\delta})(B_{1}-B_{1-\delta
})^{2}+K_{1}]+o(\delta).
\end{array}
\]
For each $P\in \mathcal{P}_\varphi$,
\[%
\begin{array}
[c]{rl}%
\frac{u(1-\delta,0)-u(1,0)}{-\delta} & =\frac{1}{-2\delta}\mathbb{E%
}[-\varphi^{\prime \prime}(B_{1-\delta})(B_{1}-B_{1-\delta})^{2}+K_{1}]+o(1)\\
& \leq \frac{1}{2\delta}E_{P}[\varphi^{\prime \prime}(B_{1-\delta})(B_{1}%
-B_{1-\delta})^{2}]+o(1)\\
& \leq E_{P}[G(\varphi^{\prime \prime}(B_{1-\delta}))]+o(1).
\end{array}
\]
Thus
\[
\sup_{P\in \mathcal{P}_\varphi}%
E_{P}[G(\varphi^{\prime \prime}(B_{1}))]=\partial_tu(1,0)\leq \inf
_{P\in \mathcal{P}_\varphi}E_{P}[G(\varphi^{\prime \prime}(B_{1}))].
\] Consequently,  for $P\in \mathcal{P}_\varphi$, we have $\partial_tu(1,0)=E_{P}[G(\varphi^{\prime \prime}(B_{1}))]$ .

\end {proof}

\begin{remark}
In \cite{HJ}, the authors used the similar idea to obtain the variation equation for the cost functional associated with the stochastic
recursive optimal control problem.
\end{remark}

\begin {corollary} Let $H\in C^2(\mathbb{R})$ with polynomial
growth satisfy, for some $\rho>0$, \[\frac{x}{2}H'(x)-G(H''(x))=\rho H(x).\] Then we have \[\mathbb{E}[H(B_1)]=0.\]
\end {corollary} The proof is  immediate from Proposition \ref {main}. Actually, for $P\in \mathcal{P}$ such that $E_P[H(B_1)]=\mathbb{E}[H(B_1)]$, we have
\[\rho\mathbb{E}[H(B_1)]=\rho E_P[H(B_1)]=E_P[\frac{B_1}{2}H'(B_1)-G(H''(B_1))]=0.\] Below we give a direct proof.

\begin {proof} Let $X^x_t=e^{-\frac{t}{2}}x+\int_0^te^{-\frac{1}{2}(t-s)}dB_s.$ Applying It\^o's formula to $e^{\rho t} H(X^x_t)$,  we have
\begin {eqnarray*} e^{\rho t} H(X^x_t)&=&H(x)+\int_0^te^{\rho s}(\rho H(X^x_s)-\frac{1}{2}X^x_s H'(X^x_s)+G(H''(X^x_s)))ds\\
& &+\int_0^te^{\rho s}H'(X^x_s)d B_s+ \frac{1}{2}\int_0^te^{\rho s}H''(X^x_s)d \langle B\rangle_s-\int_0^te^{\rho s}G(H''(X^x_s))ds.
\end {eqnarray*} So $\mathbb{E}[H(X^x_t)]=e^{-\rho t} H(x)$ and \[\mathbb{E}[H(B_1)]=\lim_{t\rightarrow\infty}\mathbb{E}[H(X^x_t)]=0.\]

\end {proof}


\renewcommand{\refname}{\large References}{\normalsize \ }

\end{document}